\documentclass[12pt]{article}

\usepackage{datetime}

\usepackage[all]{xy}

\usepackage{amssymb, amsmath, amsthm,eucal}

\usepackage{mathptmx}
\usepackage[scaled=.90]{helvet}
\usepackage{courier}

\usepackage[pdftex]{hyperref}

\hypersetup{
  colorlinks = true,
  linkcolor = black,
  urlcolor = blue,
  citecolor = black,
}

\numberwithin{equation}{section}

\renewcommand{\H}{\mathrm{H}}

\renewcommand{\SS}{\mathbb{S}}

\newcommand{\A}{\mathrm{A}}
\newcommand{\B}{\mathrm{B}}
\newcommand{\E}{\mathrm{E}}
\newcommand{\I}{\mathrm{I}}
\newcommand{\K}{\mathrm{K}}

\renewcommand{\L}{\mathrm{L}}

\newcommand{\CC}{\mathbb{C}}
\newcommand{\FF}{\mathbb{F}}
\newcommand{\QQ}{\mathbb{Q}}

\newcommand{\ZZ}{\mathbb{Z}}

\newcommand{\Hom}{\mathrm{Hom}}

\newcommand{\BO}{\mathrm{BO}}

\newcommand{\SL}{\mathrm{SL}}
\newcommand{\SO}{\mathrm{SO}}

\newcommand{\J}{\mathrm{J}}
\newcommand{\cokerJ}{\mathrm{coker}(\J)}
\newcommand{\JO}{\mathrm{JO}}
\newcommand{\JSO}{\mathrm{JSO}}
\newcommand{\JSpin}{\mathrm{JSpin}}

\newcommand{\KO}{\mathrm{KO}}
\newcommand{\KU}{\mathrm{KU}}

\newcommand{\BU}{\mathrm{BU}}
\newcommand{\MU}{\mathrm{MU}}

\newcommand{\RU}{\mathrm{RU}}

\newtheorem{theorem}{Theorem}[section]
\newtheorem{conjecture}[theorem]{Conjecture}
\newtheorem{proposition}[theorem]{Proposition}
\newtheorem{corollary}[theorem]{Corollary}

\theoremstyle{definition}
\newtheorem{definition}[theorem]{Definition}
\newtheorem{example}[theorem]{Example}
\newtheorem{remark}[theorem]{Remark}
\newtheorem{digression}[theorem]{Digression}

\numberwithin{equation}{section}

\title{\bf The chromatic filtration\\ of the Burnside category} 

\author{Markus Szymik}

\date{March 2012}

\begin{document}

\maketitle

\begin{abstract}
\noindent
The Segal map connects the Burnside category of finite groups to the stable homotopy category of their classifying spaces. The chromatic filtrations on the latter can be used to define filtrations on the former. We prove a related conjecture of Ravenel's in some cases, and present counterexamples to the general statement.
\end{abstract}

\thispagestyle{empty}


\section*{Introduction}

The Segal map connects the Burnside category of finite groups to the stable homotopy category of their classifying spaces. The purpose of this writing is to study the interplay with the chromatic filtration of the stable homotopy category. For the purposes of the introduction, we will focus on the more familiar Burnside rings.

Let~$G$ be a finite group, and let~$\A(G)$ be its Burnside ring of isomorphism classes of virtual finite~$G$-sets. This is an algebraic invariant of~$G$ which plays a fundamental role for its representation theory. See~\cite{Solomon},~\cite{Dress},~\cite{tomDieck}, and~\cite{Bouc} for some basic information on these rings. 

The Burnside rings are related to topology by the Segal map, which is a natural map
\begin{equation*}
  \A(G)\longrightarrow[\,\Sigma^\infty_+(\B G)\,,\,\SS\,]
\end{equation*}
from the Burnside ring to the ring of stable homotopy classes of maps from the suspension spectrum of the classifying space~$\B G$ of~$G$ to the sphere spectrum~$\SS$. The subscript~$+$ indicates that a disjoint base point has been added. By the completion theorem for stable cohomotopy, formerly known as the \hbox{Segal} conjecture, this map is iso\-morphic to the completion
\begin{equation*}
	\A(G)\to\widehat \A(G)
\end{equation*}
at the augmentation ideal of the Burnside ring. See~\cite{Carlsson} and the references therein. 

The target of the Segal map is universal in the following sense: if~$R$ is any ring spectrum, composition with the
unit~\hbox{$\SS\rightarrow R$} induces a ring map
\begin{equation*}
  \A(G)\longrightarrow[\,\Sigma^\infty_+(\B G)\,,\,R\,].
\end{equation*}
One would like to think that these are maybe easier to understand than the universal case. For example, if~$R$ is complex oriented as an algebra over the complex Thom spectrum~$\MU$, these maps and in particular their targets have been investigated by Hopkins, Kuhn, and Ravenel in~\cite{Ravenel:morava}, \cite{Kuhn:modp}, \cite{Kuhn:morava}, \cite{Kuhn:survey}, \cite{Hopkins},~\cite{HKR1}, and~\cite{HKR2}. 

In this paper, we turn our attention to ring spectra~$R$ which are in some sense closer to the sphere spectrum~$\SS$ itself:
if~$E$ is any spectrum, the Bousfield~$E$-localization~$\L_E\SS$ is a ring spectrum, and there results a ring map
\begin{equation*}
  \A(G)\longrightarrow[\,\Sigma^\infty_+(\B G)\,,\,\L_E\SS\,].
\end{equation*}
Of particular interest in this context are the localizations~\hbox{$\L_n=\L_{\E(n)}$} with respect to the Johnson-Wilson spectra~$\E(n)$. The spectrum~$\E(0)=\H\QQ$ is the rational Eilenberg-Mac Lane spectrum. But, if~$n\geqslant1$, these spectra not only depend on the integer~\hbox{$n\geqslant0$} but also on a prime~$p$, which is commonly omitted from the notation. In the case~$n=1$, the spectrum~$\E(1)$ is one of~\hbox{$p-1$} summands~(which are all stably equivalent up to suspensions) of~$p$-local complex~K-theory. The localizations~\hbox{$\L_n$} filter the stable homotopy category into layers which show periodic phenomena; this is called the chromatic filtration. See~\cite{Ravenel:Localization} and Section~\ref{sec:chromatic} for more background information. 

As Morava has put it, ``Now that Segal's conjectures have been proven, it would be very interesting to understand their relation to the chromatic filtration''~\cite{Morava}.

Ravenel, in~\cite{Ravenel:Conjecture}, has posed the problem to determine the kernels of the localization maps
\begin{equation*}
	\widehat \A(G)\cong[\,\Sigma^\infty_+(\B G)\,,\,\SS\,]\longrightarrow[\,\Sigma^\infty_+(\B G)\,,\,\L_n\SS\,],
\end{equation*}
and he has conjectured that these contain all virtual finite~$G$-sets~$X$ such that the fixed point set~$X^U$ has virtual cardinality~$0$ for all subgroups~$U$ of~$G$ which are generated by at most~$n$ elements.

The main aim of this note is to show that, while this conjecture does hold in the rational case~\hbox{$n=0$}, it fails already for the next case~$n=1$ because of some particular features at the prime~2. Before we do so, we put this into the more general perspective of the Burnside category of finite groups. Section~\ref{sec:BurnsideSegal} reviews this and the Segal map in a way that fits the purposes of the following text. In the following Sections~\ref{sec:chromatic} and~\ref{sec:Ravenel}, we will introduce the chromatic filtrations and the Ravenel filtration of the Burnside category, respectively. This allows us to phrase the conjecture slightly more conceptually. Along the way, we will prove a couple of useful properties of these filtrations, and we will also see that (a generalization of) the rational case~$n=0$ of the conjecture holds. Section~\ref{sec:n=1} contains our positive results for the next case~$n=1$ of Ravenel's conjecture, which holds if one considers odd primes only. We also mention that we might just as well have used localization with respect to the first Morava~K-theory in this case. The final Section~\ref{sec:counterexample} presents examples and counterexamples for the case~\hbox{$n=1$} at the prime~\hbox{$p=2$}, based on earlier work of Laitinen~\cite{Laitinen} and Feshbach~\cite{Feshbach}. In particular, we will discuss the Klein group~$V_4$, the alternating group~$A_4$, and the dihedral group~$D_8$.

The existence of bad groups~\cite{Kriz},~\cite{Kriz+Lee} prevents us from extending our discussion to higher chromatic levels~$n\geqslant 2$, see Remark~\ref{rem:higher_n}.


\section{The Burnside category and the Segal map}\label{sec:BurnsideSegal}

A finite group~$G$ has a classifying space~$\B G$, and~$\Sigma^\infty_+(\B G)$ will denote its corresponding suspension spectrum, with a disjoint base point added. A homomorphism~\hbox{$G\to H$} of groups induces a map~$\B G\to \B H$ between their classifying spaces, and also a stable map~\hbox{$\Sigma^\infty_+(\B G)\to\Sigma^\infty_+(\B H)$}. This construction extends over the Burnside category, which is defined as follows, see~\cite[Section 4]{Adams:Segal},~\cite[\S 9]{Adams+Gunawardena+Miller} and also~\cite{May}.

\subsection{The Burnside category}

The objects of the Burnside category are the finite groups. The set~$\A(G,H)$ of morphisms from~$G$ to~$H$ is the Grothendieck group on the abelian monoid of isomorphisms classes of finite~$(G,H)$-bisets which are~$H$-free. 

\begin{example}
	If~\hbox{$H=e$} is the trivial group, them the abelian group~$\A(G,e)$ is the~(underlying abelian group of) the Burnside ring~$\A(G)$ of~$G$. 
\end{example}

Finite~$(G,H)$-bisets can be identified with finite sets with an action of the product group~\hbox{$G\times H$}, because the actions of~$G$ and~$H$ are required to commute: the formula is~\hbox{$(g,h)s=gsh^{-1}$}. This embeds the abelian group~$\A(G,H)$ into the Burnside ring~$\A(G\times H)$ of the product group. If~$U$ is a subgroup of~$G\times H$, then the map which sends a set~$S$ with an action of~$G\times H$ to the set~$S^U$ of its~$U$-fixed points extends to give~a homomorphism
\begin{equation*}
	\Phi^U\colon\A(G\times H)\longrightarrow\ZZ,
\end{equation*}
the~$U$-mark homomorphism, and its restriction to~$\A(G,H)$ will be denoted by the same symbol. The~$e$-mark~$\Phi^e$ is usually referred to as the augmentation.

\begin{remark}\label{rem:vanishing}
The~$U$-mark~$\Phi^U$ vanishes on~$\A(G,H)$ as soon as~$U\cap H\not=e$, because~$H$ acts freely on the classes in~$\A(G,H)$.
\end{remark}

Composition in the Burnside category is induced by sending a~$(G,H)$-biset~$S$ and an~$(H,I)$-biset~$T$ to~$S\times_H T$. The identity in~$\A(G,G)$ is~$G$ with the actions of~$G$ by multiplication from the left and from the right. The category of finite groups is embedded into the Burnside category: if~$\phi\colon G\to H$ is a homomorphism of groups, then~$H$ is an~$H$-free finite~$(G,H)$-biset, where~$G$ acts on~$H$ from the left via~$\phi$, and where~$H$ acts on it via the right multiplication. Similarly, but contravariantly, if~$\theta\colon H\to G$ is injective, then~$G$ is an~$H$-free finite~$(G,H)$-biset, where~$G$ acts via left multiplication, and~$H$ acts freely from the right via~$\theta$. Morphisms of these two types generate the Burnside category.

\subsection{The Segal maps} 

The Segal maps
\begin{equation}\label{eq:Segal}
	\alpha\colon\A(G,H)\longrightarrow[\,\Sigma^\infty_+(\B G)\,,\,\Sigma^\infty_+(\B H)\,]
\end{equation}
define an additive functor from the Burnside category to the stable homotopy category. For example, if~$S$ is an~$H$-free finite~$(G,H)$-biset which represents a class in~$\A(G,H)$, this induces an~$H$-cover of the form~\hbox{$\E G\times_GS\to \E G\times_G(S/H)$} and this in turn is classified by a map~\hbox{$\E G\times_G(S/H)\to \B H$}. Stabilization and precomposition with the stable transfer associated with the cover~\hbox{$\E G\times_G(S/H)\to \B G$} gives the value 
\begin{equation*}
	\Sigma^\infty_+(\B G)\longrightarrow
	\Sigma^\infty_+(\E G\times_G(S/H))\longrightarrow
	\Sigma^\infty_+(\B H)
\end{equation*}
of~$\alpha$ on~$S$.

As a consequence of Carlsson's affirmative solution~\cite{Carlsson} to the Segal conjecture, the Segal map induces an isomorphism between the completion of~$\A(G,H)$ with respect to the augmentation ideal of the Burnside ring~$\A(G)$, and the group~$[\,\Sigma^\infty_+(\B G)\,,\,\Sigma^\infty_+(\B H)\,]$ of stable maps between their classifying spaces. In fact, this statement can be generalized to give a description of the function spectrum (and to compact Lie groups~$H$). See~\cite{Lewis+May+McClure},~\cite{May}, and~\cite{May+Snaith+Zelewski}. However, we will restrict our attention to the algebra in the following.

\subsection{A linear analogue of the Burnside category}

There is a linear analogue of the Burnside category, where finite sets are replaced by finite-dimensional complex vector spaces throughout. For lack of reference, we will have to spell this out in detail.

The objects of this linear category are again the finite groups. The set~$\RU(G,H)$ of morphisms from~$G$ to~$H$ is the Grothendieck group on the abelian monoid of isomorphisms classes of finite-dimensional~$(\CC G,\CC H)$-bimodules which are free over~$\CC H$. For example, the abelian group~$\RU(G,e)$ is the~(underlying abelian group of) the complex representation ring~$\RU(G)$ of~$G$. On the other hand, finite-dimensional~$(\CC G,\CC H)$-bimodules can be identified with finite-dimensional complex vector spaces with a complex linear action of the product group~\hbox{$G\times H$}, because the actions of~$G$ and~$H$ are required to commute: the formula is again~\hbox{$(g,h)v=gvh^{-1}$}. This embeds the abelian group~$\RU(G,H)$ into the representation ring~$\RU(G\times H)$ of the product group. In contrast to the case of the Burnside rings, representation rings behave well with respect to products:~$\RU(G\times H)\cong\RU(G)\otimes\RU(H)$. Composition in this category is induced by the sending a~$(\CC G,\CC H)$-bimodule~$V$ and an~$(\CC H,\CC I)$-bimodule~$W$ to the tensor product~$V\otimes_{\CC H}W$. The identity in~$\RU(G,G)$ is the biregular representation~$\CC G$. This category receives a linearization functor from the Burnside category, which is the identity on objects, and sends a~$(G,H)$-biset~$S$ to the permutation~$(\CC G,\CC H)$-bimodule~$\CC S$.

\begin{remark}\label{rem:RUiso}
As has already been mentioned, in contrast to the case of the Burnside rings, representation rings behave well with respect to products. In the same vein, the homomorphism
\begin{equation*}
	\RU(G,H)\longrightarrow\RU(G),\,V\longmapsto V\otimes_{\CC H}\CC,
\end{equation*}
induced by the unique homomorphism~$H\to e$, is always an isomorphism. The inverse is given by
\begin{equation*}
	\RU(G)\longrightarrow\RU(G,H),\,W\longmapsto W\otimes_\CC\CC H,
\end{equation*}
where~$G$ acts on~$W\otimes_\CC\CC H$ only on the left hand side. To see that this is the inverse, use that there is an isomorphism~$V\otimes_{\CC H}\CC\otimes_\CC\CC H\cong V$ if~$V$ is free over~$\CC H$. Because of this fact, this linear category consists essentially only of representation rings.
\end{remark}


\section{The chromatic filtration of the Burnside category}\label{sec:chromatic}

\subsection{The chromatic filtration of the stable homotopy category}

Before we define the chromatic filtration of the Burnside category, let us first briefly review the chromatic filtration~\cite{Ravenel:Localization} on the stable homotopy category. For each prime number~$p$ and integer~$n\geqslant0$, let~$\E(n)$ be the Johnson-Wilson spectrum. The coefficients are given by~$\pi_*\E(n)\cong\ZZ_{(p)}[v_1,\dots,v_{n-1},v_n^{\pm1}]$, where~$v_n$ has degree~\hbox{$2(p^n-1)$}. Bousfield localization~\cite{Bousfield:Localization} of a spectrum~$X$ with respect to~$\E(n)$ will be denoted by~$X\to\L_nX$. There are natural transformations~$\L_nX\to\L_{n-1}X$ for all spectra~$X$, so that these functors determine a tower
\begin{equation*}
\dots\longrightarrow\L_nX\longrightarrow\dots\longrightarrow\L_1X\longrightarrow\L_0X,
\end{equation*}
which defines the chromatic filtration (at~$p$) of~$X$. We will denote by~$\K(n)$ the~$n$-th Morava K-theory spectrum. The coefficients are~$\pi_*\K(n)\cong\FF_{p}[v_n^{\pm1}]$. Using these, Bousfield localization~$\L_{n-1}X$ with respect to~$\E(n-1)$ is equivalent to Bousfield localization with respect to~$\K(0)\vee\dots\vee\K(n-1)$, and~$\L_{n-1}X$ is~$\K(n)$-acyclic. This implies that the difference between~$\L_nX$ and~$\L_{n-1}X$ is described by the chromatic Hasse square
\begin{center}
  \mbox{ 
    \xymatrix{
      \L_n X\ar[r]\ar[d] & \L_{\K(n)}X\ar[d] \\
      \L_{n-1}X\ar[r]&\L_{n-1}\L_{\K(n)}X,
    } 
  }
\end{center}
which is a pullback square for all~$X$. The localizations~$\L_{\K(n)}X$ are therefore the basic pieces from which one tries to assemble~$X$.

\subsection{The chromatic filtration of the Burnside category}

The chromatic filtrations on the stable homotopy category restricts to the chromatic filtrations on the Burnside category: Let~$G$ and~$H$ be finite groups, and let~$\A(G,H)$ be the group of morphisms~$G\to H$ in the Burnside category. For an integer~\hbox{$n\geqslant0$}, and a prime~$p$, let~$\I_{n,p}(G, H)$ be the kernel of the composition
\begin{equation*}
	\A(G,H)\overset{\alpha}{\longrightarrow}[\,\Sigma^\infty_+(\B G)\,,\,\Sigma^\infty_+(\B H)\,]\overset{\L_n}{\longrightarrow}[\,\L_n\Sigma^\infty_+(\B G)\,,\,\L_n\Sigma^\infty_+(\B H)\,]
\end{equation*}
of the Segal map~$\alpha$ with the Bousfield localization~$\L_n$ at the prime~$p$. Note that a localization~$\L_E$ induces an isomorphism
\begin{equation*}
[\,\L_E\Sigma^\infty_+(\B G)\,,\,\L_E\Sigma^\infty_+(\B H)\,]\cong[\,\Sigma^\infty_+(\B G)\,,\,\L_E\Sigma^\infty_+(\B H)\,],
\end{equation*}
so that there is never a need to localize the source.

Since there are natural transformations~\hbox{$\L_n\to \L_{n-1}$}, these abelian groups form a decreasing sequence
\begin{equation}
  \A(G,H)\supseteq \I_{0,p}(G,H)\supseteq \I_{1,p}(G,H)\supseteq \I_{2,p}(G,H)\supseteq\dots\,.
\end{equation}
In order to free ourselves from the grip of the prime~$p$, we may also consider the subgroups
\begin{equation*}
	\I_n(G,H)=\bigcap_p\,\I_{n,p}(G,H)
\end{equation*}
of~$\A(G,H)$ which no longer depend on~$p$. They also form a decreasing sequence
\begin{equation}\label{eq:chromatic_filtration}
  \A(G,H)\supseteq \I_0(G,H)\supseteq \I_1(G,H)\supseteq \I_2(G,H)\supseteq\dots\,.
\end{equation}

\begin{definition}
	The filtration~\eqref{eq:chromatic_filtration} on the morphism groups~$\A(G,H)$ of the Burnside category will be referred to as the {\it chromatic filtration}. 
\end{definition}

If~$H=e$ is the trivial group, so that~$\A(G,e)=\A(G)$ is the Burnside ring of~$G$, then we will simply write~$\I_n(G)=\I_n(G,e)$, of course.

Before we discuss the case~$n=0$, let us first document a useful property of the chromatic filtration.

\begin{proposition}\label{prop:I-restriction}
	For each~$n\geqslant0$, the subgroups~$\I_n(G,H)\subseteq\A(G,H)$ form a two-sided ideal in the Burnside category. 
\end{proposition}

This means: if~$R$,~$S$, and~$T$ are composable morphisms in the Burnside category such that~$RST$ is defined, and~$S$ lies in~$\I_n$, then so do~$RS$ and~$ST$.

\begin{proof}
	This is clear from the functoriality both of the Segal map and of the localization: pre- and post-composition with zero-maps leads to zero-maps.
\end{proof}

\begin{corollary}\label{cor:I-restriction}
All the chromatic subgroups~$\I_n(G)$ of the Burnside rings~$\A(G)$ are ideals, and these are functorial in the group: if~\hbox{$\phi\colon G\to H$} is a morphism of groups, then the morphism~\hbox{$\phi^*\colon \A(H)\to \A(G)$} of rings which is defined by restriction along~$\phi$ respects the filtration: it maps the ideal~$\I_n(H)$ to the ideal~$\I_n(G)$.
\end{corollary}

\subsection{The rational case}

As a reassuring exercise, let us check that we completely understand the zeroth case~\hbox{$n=0$}.

\begin{example}\label{ex:I_0}
	We have~$\I_{0,p}(G,H)=\I_0(G,H)$ for all primes~$p$: the spectrum~$\L_0\SS$ is the rational Eilenberg-MacLane spectrum~$\H\QQ$, and therefore independent of~$p$. This also shows that there is an isomorphism
\begin{equation*}
  [\,\L_0\Sigma^\infty_+(\B G)\,,\,\L_0\Sigma^\infty_+(\B H)\,]=\Hom_\QQ(\,\H\QQ_*(\B G)\,,\,\H\QQ_*(\B H)\,)\cong\QQ.
\end{equation*}
The description of the Segal map in Section~\ref{sec:BurnsideSegal} shows that, under this identification, the induced homomorphism~\hbox{$\A(G,H)\to\QQ$} sends a virtual finite~$H$-free~$(G,H)$-biset~$S$ to the virtual cardinality of~$S/H$. Therefore, the subgroup~$\I_0(G,H)$ of the group~$\A(G,H)$ consist of those virtual finite~$H$-free~$(G,H)$-bisets~$S$ in~$\A(G,H)$ such that~$S/H$ has virtual cardinality~$0$. In particular, the ideal~$\I_0(G)$ 
consist of those virtual finite~$G$-sets~$S$ in~$\A(G)$ such that~$S$ has virtual cardinality~$0$: this is the augmentation ideal of the Burnside ring~$\A(G)$.
\end{example}

We will see later that already in next case~$n=1$ the chromatic ideal is no longer as easy to describe.


\section{The Ravenel filtration of the Burnside category}\label{sec:Ravenel}

\subsection{The Ravenel filtration}

Let again~$G$ and~$H$ be finite groups, and let~$\A(G,H)$ be the abelian group of morphisms~\hbox{$G\to H$} in the Burnside category as before. For an integer~\hbox{$n\geqslant0$}, let~$\J_n(G,H)$ be the subgroup of~$\A(G,H)$ which consists of the virtual finite~$H$-free~$(G,H)$-bisets~$S$ such that the~$U$-marks~\hbox{$\Phi^U(S/H)$} of the virtual finite~$G$-set~$S/H$ vanish for all subgroups~$U\subseteq G$ which are generated by at most~$n$ elements. 

\begin{remark}
	If~$H=e$ is the trivial group, so that we are dealing with the Burnside rings, the ideals~\hbox{$\J_n(G)=\J_n(G,e)$} are the ones implicit in~\cite{Ravenel:Conjecture}, and the definition given above can rephrased to say that~$\J_n(G,H)$ is the preimage of~$\J_n(G)$ under the homomorphism~\hbox{$\A(G,H)\to\A(G)$} induced by the unique homomorphism~$H\to e$ of groups. From several possible choices of generalization of this special case to the bivariant situation, the one presented above has been chosen, because it fits our purposes best. 
\end{remark}
 
For varying~$n$, the ideals~$\J_n(G,H)$ evidently form a decreasing sequence
\begin{equation}\label{eq:Ravenel_filtration}
  \A(G,H)\supseteq \J_0(G,H)\supseteq \J_1(G,H)\supseteq \J_2(G,H)\supseteq\dots,
\end{equation}
and by the finiteness of the set of subgroups of~$G$, there is always an integer~$n$, depending on~$G$, such that~\hbox{$\J_n(G,H)=0$}. In fact, a rather crude upper bound for~$n$ is the order of~$G$~(minus~$1$). This shows
that this filtration has the Hausdorff property.
\begin{equation}\label{eq_J_n_intersection_is_null}
	\bigcap_n\,\J_n(G,H)=0
\end{equation}

\begin{definition}
	The filtration~\eqref{eq:Ravenel_filtration} on the group~$\A(G,H)$ will be referred to as the {\it Ravenel filtration}. 
\end{definition}

Note that the Ravenel filtration does not depend on the choice of a prime~$p$. 

\subsection{Conjectures} 

With the notions introduced so far, Ravenel's conjecture in~\cite{Ravenel:Conjecture} can be reformulated a follows.

\begin{conjecture}\label{conj:Ravenel}
	For all finite groups~$G$, and all integers~$n\geqslant0$, the Ravenel ideals~$\J_n(G)$ of the Burnside rings~$\A(G)$ lie in the chromatic ideals~$\I_n(G)$.
\end{conjecture}

For the purposes of later reference, we also formulate the following stronger statement. 

\begin{conjecture}\label{conj:strong}
For all finite groups~$G$ and~$H$, and integers~$n\geqslant0$, the Ravenel subgroups~$\J_n(G,H)$ of~$\A(G,H)$ lie inside the chromatic groups~$\I_n(G,H)$. 
\end{conjecture}

Since already the original Conjecture~\ref{conj:Ravenel} turns out to be indefensible in general, so is Conjecture~\ref{conj:strong}. However, if~$G=e$ is the trivial group, then the strong form of Conjecture~\ref{conj:strong} is true for all~$H$ and~$n$, since~$\A(e,H)\cong\ZZ$ and~\hbox{$\I_0(e,H)=0=\J_0(e,H)$} in this case. This suggests that the case~$H=e$ is the gist of the matter here, and we will later focus on it.

\subsection{The ideal property} 

Before we examine the first few steps of the Ravenel filtration, let us document the property of the Ravenel filtration which corresponds to Proposition~\ref{prop:I-restriction} for the chromatic filtration.

\begin{proposition}
	For each~$n\geqslant0$, the subgroups~$\J_n(G,H)\subseteq\A(G,H)$ form a two-sided ideal in the Burnside category. 
\end{proposition}

\begin{proof}
	Let~$R\colon F\to G$,~$S\colon G\to H$, and~$T\colon H\to I$ be composable morphisms in the Burnside category, and assume that~$S$ lies in~$\J_n(G,H)$. Then we have to show that the compositons~$R\times_GS$ and~$S\times_HT$ are in~$\J_n$ as well.
	
	Let us start with~$S\times_HT$. There are two cases to consider. On the one hand, if~$T$ is given by a group morphism~$H\to I$, then~$T=I$, and therefore
\begin{equation*}
	S\times_HT/I=S\times_HI/I\cong S/H.
\end{equation*} 
Of course, this implies that
\begin{equation*}
	\Phi^U(S\times_HT/I)=\Phi^U(S/H)=0
\end{equation*} 
for all subgroups~$U$ of~$G$ which are generated by at most~$n$ elements. On the other hand, if~$T$ is given by a subgroup~$I\leqslant H$, then~$T=H$, and therefore
\begin{equation*}
	S\times_HT/I=S\times_HH/I\cong S/I. 
\end{equation*} 
There is a~$G$-equivariant map~$S/I\to S/H$, and this induces a map between the~$U$-fixed points, for all subgroups~$U$ of~$G$ which are generated by at most~$n$ elements. Because the latter are virtually empty by assumption, so are the former. 
	
Let us finally deal with the other case~$R\times_GS$ as well. There are again two cases to consider. On the one hand, if~$R$ is given by a morphism~$F\to G$ of groups, then~$R=G$, and therefore
\begin{equation*}	
	R\times_GS=G\times_GS\cong S. 
\end{equation*} 	
If~$U$ is a subgroup of~$F$ which is generated by at most~$n$ elements, so is the image~$\overline U$ of~$U$ in~$G$. This implies
\begin{equation*} 
	\Phi^U(R\times_GS)=\Phi^{\overline U}(S)=0
\end{equation*} 
by assumption. On the other hand, if~$R$ is given by a subgroup~$G\leqslant F$, then~$R=F$, and therefore~$R\times_GS=F\times_GS$ it the induced~$F$-set. If we write~$\overline S=S/H$, it remains to be shown that~$\Phi^U(\overline S)=0$ for all subgroups~$U\leqslant G$ which are generated by at most~$n$ elements implies~\hbox{$\Phi^V(F\times_G\overline S)=0$} for all subgroups~$V\leqslant F$ which are generated by at most~$n$ elements. Recall~$G\leqslant F$. Therefore, if~$V\not\leqslant G$, then~$\Phi^V(F\times_G\overline S)=0$ anyway, since~$F\times_G\overline S$ is a linear combination of classes represented by~$F\times_GG/H\cong F/H$ with~\hbox{$H\leqslant G$}, and these~$F$-sets have no~$V$-fixed points if~$V\not\leqslant G$. If, on the contrary,~$V\leqslant G$, then~$\Phi^V(F\times_G\overline S)$ is determined by the restriction of the induced~$F$-set~$F\times_G\overline S$ back to~$G$, and this is just a multiple of~$\overline S$. (The multiplicity is the index of~$G$ in~$F$.) Therefore~$\Phi^V(\overline S)=0$ implies~\hbox{$\Phi^V(F\times_G\overline S)=0$} also in this case.
\end{proof}

\subsection{The cases \texorpdfstring{$n=0$}{n=0} and \texorpdfstring{$n=1$}{n=1}}

Let us now examine the first two steps of the Ravenel filtration.

\begin{example}\label{ex:J_0}
	Since in every group~$G$ there exists a unique subgroup generated by no elements, the trivial subgroup~$e$, the subgroup~$\J_0(G,H)$ of~$\A(G,H)$ consist of those virtual finite~$H$-free~$(G,H)$-bisets~$S$ in~$\A(G,H)$ such that~$S/H$ has virtual cardinality~$0$. In particular, the ideal~$\J_0(G)$ is the augmentation ideal of the Burnside ring~$\A(G)$. 
\end{example}

\begin{theorem}\label{thm:n=0}
	Ravenel's Conjecture~\ref{conj:Ravenel}, even in the strong form of Conjecture~\ref{conj:strong}, holds for~$n=0$.
\end{theorem}

\begin{proof}
	This follows immediately from our Examples~\ref{ex:I_0} and~\ref{ex:J_0} above.
\end{proof}

The next case~$n=1$ of the Ravenel filtration can be understood in terms of the linearization functor.

\begin{proposition}\label{prop:J_1}
  For all finite groups~$G$ and~$H$, the subgroup~$\J_1(G,H)$ of~$\A(G,H)$ is equal to the kernel of the linearization
  map~\hbox{$\A(G,H)\rightarrow \RU(G,H)$}.
\end{proposition}

\begin{proof}
	Let us first check this for the case~$H=e$, when we only have to deal with the Burnside and representation rings. (This case is well known, of course.) If~$S$ is a virtual~$G$-set, then the character~$\chi_V$ of the associated permutation representation~$V=\CC S$ is given by~$\chi_V(g)=\Phi^C(S)$, where~$C=\langle g\rangle$ is the subgroup of~$G$ which is generated by~$g$. By definition, the subgroup~$C$ is cyclic, and if~$g$ runs through the elements of~$G$, the subgroups~$C=\langle g\rangle$ run through the set of all cyclic subgroups of~$G$. Since a virtual representation is zero if and only if its character vanishes, the kernel of the linearization map~$\A(G)\to\RU(G)$ consists of those~$S$ with~$\Phi^C(S)=0$ for all cyclic subgroups~$C$ of~$G$, which is precisely~$\J_1(G)$.
	
	The general case, where~$H\not=e$, can be reduced to what has just been proven, by means of the following commutative diagram, which reflects the naturality of the linearization.
\begin{center}
  \mbox{ 
    \xymatrix{
      \A(G,H)\ar[r]\ar[d] & \RU(G,H)\ar[d] \\
      \A(G)\ar[r]&\RU(G)
    } 
  }
\end{center}
By definition, an element~$S$ in~$\A(G,H)$ lies in~$\J_1(G,H)$ if and only if~$S/H$ lies in~$\J_1(G)$. By what has been shown above, this is the case if and only if the associated virtual permutation representation is zero. This argument already proves that the kernel of the linearization is contained in~$\J_1(G,H)$. The converse inclusion would follow if the downward arrow on the right hand side were injective. But, this is even an isomorphism by Remark~\ref{rem:RUiso}.	
\end{proof}


\section{Positive results for the first chromatic layer}\label{sec:n=1}

In this section, we will see that Ravenel's Conjecture~\ref{conj:Ravenel} is almost true for~$n=1$: it holds if we consider odd primes only, see Theorem~\ref{thm:E(1)}. This result concerns the original form of the conjecture only, see Remark~\ref{rem:bivariant} for a hint towards the bivariant statement. We will also comment on the effect of changing the target from the~$\E(1)$-local sphere to the~$\K(1)$-local sphere.

\subsection{Fibration sequences}

Adams, Baird, and Ravenel have shown that there is a fibration sequence
\begin{equation}\label{eq:fibration_sequence1}
	\L_1\SS\longrightarrow F\longrightarrow\Sigma^{-1}\mathrm{H}\mathbb{Q},
\end{equation}
where~$F$ is the fibre in another fibration sequence
\begin{equation}\label{eq:fibration_sequence2}
	F\longrightarrow \KU_{(p)}\longrightarrow \KU_{(p)}.
\end{equation}
The first sequence~\eqref{eq:fibration_sequence1} is~\cite[Theorem~4.3]{Bousfield:Localization}, and holds for all primes; the second one~\eqref{eq:fibration_sequence2} is~\cite[Corollary~4.4]{Bousfield:Localization}, and holds for odd primes only. The spectrum~$\KU_{(p)}$ is the~$p$-localization of~$\KU$, which is also a model for~$\L_1\KU$.

\begin{proposition}\label{prop:injective}
	For all odd primes, the map
	\begin{equation}\label{eq:injective}
		[\,\Sigma^\infty_+(\B G)\,,\,\L_1\SS\,]\longrightarrow[\,\Sigma^\infty_+(\B G)\,,\,\L_1\KU\,]
	\end{equation}
	induced by the unit~$\SS\to\KU$ is injective.
\end{proposition}

\begin{proof}
	The localization of the unit factors as~$\L_1\SS\to F\to\KU_{(p)}$ with maps as in~\eqref{eq:fibration_sequence1} and~\eqref{eq:fibration_sequence2}. 	Since~$\B G$ is rationally acyclic, it follows from the fibration sequence~\eqref{eq:fibration_sequence1} that we may replace~$\L_1\SS$ by~$F$. It then follows from the fibration sequence~\eqref{eq:fibration_sequence2} that the kernel of the map \eqref{eq:injective} is dominated by
\begin{equation*}
	[\,\Sigma^\infty_+(\B G)\,,\,\Sigma^{-1}\KU_{(p)}\,]=\KU^{-1}(\B G)_{(p)},
\end{equation*}
and by Atiyah's completion theorem~\cite{Atiyah},
\begin{equation*}
	\KU^n(\B G)\cong
	\begin{cases}
		\widehat \RU(G) & n\text{ even}\\
		0 & n\text{ odd,}
	\end{cases}
\end{equation*}
this group vanishes even before localization.
\end{proof}

\begin{remark}
	The preceeding result fails at~$p=2$ even for the trivial group~$G=e$, as the kernel of the map~$\pi_0\L_1\SS\to\pi_0\L_1\KU=\ZZ_{(2)}$ has order~$2$.
\end{remark}

\begin{remark}\label{rem:bivariant}
	For the extension of Atiyah's completion theorem to the bivariant case, see~\cite[\S 6]{Greenlees}. Note that it implies that~$[\,\Sigma^\infty_+(\B G)\,,\,\Sigma^n\KU\wedge\Sigma^\infty_+(\B H)\,]$ may well be non-zero if~$H\not=e$.
\end{remark}


\subsection{The main result of this section}

This is as follows.


\begin{theorem}\label{thm:E(1)}
	The ideal~$\J_1(G)$ lies in the kernel of the map
	\begin{equation*}
	\A(G)\longrightarrow[\,\Sigma^\infty_+(\B G)\,,\,\L_1\SS\,]
	\end{equation*}
	for all odd primes~$p$. 
\end{theorem}

\begin{proof}
We can use the commutativity of the following diagram.
\begin{center}
  \mbox{ 
    \xymatrix{
	\A(G)\ar[r]\ar[d] & \widehat \A(G)\ar[r]^-\cong\ar[d]
	& [\,\Sigma^\infty_+(\B G)\,,\,\SS\,]\ar[r]\ar[d] & [\,\Sigma^\infty_+(\B G)\,,\,\L_1\SS\,]\ar[d]\\
	\RU(G)\ar[r] & \widehat \RU(G)\ar[r]^-\cong 
	& [\,\Sigma^\infty_+(\B G)\,,\,\KU\,]\ar[r] & [\,\Sigma^\infty_+(\B G)\,,\,\L_1\KU\,]
    } 
  }
\end{center}
Proposition~\ref{prop:J_1} says that~$\J_1(G)$ is the kernel of the left vertical map. This implies that it lies in the kernel of the map from the left upper corner to the right lower corner. As the map on the right is injective by the previous Proposition~\ref{prop:injective}, we are done.
\end{proof}

For~$p=2$, not only this proof breaks down as indicated: it turns out that the conclusion of Theorem~\ref{thm:E(1)} need not be true for this prime. We will discuss this in the following Section~\ref{sec:counterexample}.

\subsection{Localization with respect to Morava K-theory}

Let us briefly mention the effect of replacing the localization with respect to the Adams-Johnson-Wilson spectrum~$\E(1)$ to the slightly more manageable localization with respect to Morava K-theory~$\K(1)$. Of course, by Theorem~\ref{thm:E(1)}, the ideal~$\J_1(G)$ lies in the kernel of the map
	\begin{equation*}
	\A(G)\longrightarrow[\,\Sigma^\infty_+(\B G)\,,\,\L_{\K(1)}\SS\,]
	\end{equation*}
for all odd primes~$p$ as well. And the following result implies that these are, in fact, equivalent statements.

\begin{proposition}\label{prop:former_reduction}
	The map
	\begin{equation*}
  	[\,\Sigma^\infty_+(\B G)\,,\,\L_1\SS\,]\longrightarrow[\,\Sigma^\infty_+(\B G)\,,\,\L_{\K(1)}\SS\,]
	\end{equation*}
	is injective at all primes. 
\end{proposition}

\begin{proof}
	This map is part of the Mayer-Vietors sequence
	\begin{equation*}
		[\,\Sigma^\infty_+(\B G)\,,\,\L_0\L_{K(1)}\SS\,]\to
		[\,\Sigma^\infty_+(\B G)\,,\,\L_1\SS\,]\to
		[\,\Sigma^\infty_+(\B G)\,,\,\L_{\K(1)}\SS\,]\oplus[\,\Sigma^\infty_+(\B G)\,,\,\L_0\SS\,]
	\end{equation*}
	associated with the first chromatic Hasse square 
	\begin{center}
  \mbox{ 
    \xymatrix{
      \L_1 X\ar[r]\ar[d] & \L_{\K(1)}X\ar[d] \\
      \L_{0}X\ar[r]&\L_{0}\L_{\K(1)}X.
    } 
  }
\end{center}
	Since the are equivalences~$\L_0\SS\simeq\H\QQ$ and~$\L_0\L_{\K(1)}\SS\simeq\H\QQ_p\vee\Sigma^{-1}\H\QQ_p$, and~$\B G$ is rationally contractible, the kernel of the map in question is dominated by the trivial group, and the proposition follows.
\end{proof}

\begin{remark}
	It is easy to see by example that the map in Proposition~\ref{prop:former_reduction} need not be surjective. This happens even for the trivial group~$G=e$.
\end{remark}

\subsection{Some blind alleys and dead ends}

\begin{remark}
Let us discuss how one might try to extend the arguments of this section to the case~$p=2$. There is a fibration sequence similar to~\eqref{eq:fibration_sequence2}, but here it takes the form
\begin{equation}\label{eq:KO_fibration_sequence}
	F\longrightarrow \KO_{(2)}\longrightarrow \KO_{(2)}.
\end{equation}
Consequently, in the~$p=2$ case, one would have to replace complex K-theory with real~K-theory~$\KO$ throughout the argument. However, the group~$\KO^{-1}_{(2)}(\B G)$ need not be trivial: this does not even hold for the trivial group~\hbox{$G=e$}, since in that case we have~$\KO^{-1}_{(2)}=\pi_1\KO_{(2)}=\ZZ/2$. It is here where the previous argument breaks down, and in fact the statement is wrong in this case. A counterexample is not hard to find, and will be presented in the following section.
\end{remark}

\begin{remark}\label{rem:higher_n}
Similar problems arise if one attempts to generalize the previous arguments to the higher chromatic layers~$n\geqslant2$. In that case, it seems advisable to deal with the~$\K(n)$-local sphere rather than the~$\E(n)$-local sphere. For that, there is a fibration sequence
\begin{equation*}
	\L_{\K(n)}\SS\longrightarrow\E_n^{\mathrm{h}K}\longrightarrow\E_n^{\mathrm{h}K}
\end{equation*}
for a certain closed subgroup~$K$ of the Morava stabilizer group~$\mathbb{G}_n$ which acts on the Lubin-Tate spectrum~$\E_n$ such that there is an equivalence~$\L_{\K(n)}\SS\simeq \E_n^{\mathrm{h}\mathbb{G}_n}$, see~\cite[Proposition 8.1]{Devinatz+Hopkins}, and one may wonder whether the composition
\begin{equation*}
	\A(G)\to
	[\,\Sigma^\infty_+(\B G)\,,\,\SS\,]\to
	[\,\Sigma^\infty_+(\B G)\,,\,\L_{\K(n)}\SS\,]\to
	[\,\Sigma^\infty_+(\B G)\,,\,\E_n^{\mathrm{h}K}\,]
\end{equation*}
has the properties which were used crucially in this section: it has to kill~$\J_n(G)$ and it must be injective. One way to approach both of these is the generalized character theory of Hopkins, Kuhn, and Ravenel mentioned in the introduction. This is a means to detect elements, for example in~$[\,\Sigma^\infty_+(\B G)\,,\,\E_n\,]=\E_n^0(\B G)$. 


\begin{proposition}
	The generalized characters with respect to $\E_n$-theory of the elements in the ideal~$\J_n(G)$ vanish.
\end{proposition}

\begin{proof}
	Let $S$ be an element in the ideal~$\J_n(G)$. Recall from~\cite{HKR2} that its generalized character, say~$\chi_S$, can be evaluated on $n$-tuples of commuting elements of $G$, each of which has order a power of~$p$. Or, equivalently, on the morphisms~$\alpha\colon\ZZ_p^n\to G$ of groups. By~\cite[Example 6.16]{HKR2}, the value $\chi_S(\alpha)$ is the number of fixed points in $S$ of the image of $\alpha$. Since the image is a finite quotient of~$\ZZ_p^n$, it is clearly generated by at most~$n$ elements, so that~$S\in\J_n(G)$ implies that~$\chi_S(\alpha)$ is zero for all~$\alpha$.
\end{proof}

Unfortunately, the vanishing of the generalized characters of the elements in the ideal~$\J_n(G)$ does not imply that these elements are zero in~$\E_n^0(\B G)$. This is known to fail for the so-called bad groups, where the Morava K-theory and the~E-theories are not concentrated in even degrees or may contain~$p$-torsion. See~\cite{Kriz} and~\cite{Kriz+Lee} for examples of such groups. This implies that the generalized character map on~$\E_n^0(\B G)$ is not injective for these groups, and it also prevents us from generalizing our proof in order to show that the homomorphism
\begin{equation*}
	[\,\Sigma^\infty_+(\B G)\,,\,\L_{\K(n)}\SS\,]\to
	[\,\Sigma^\infty_+(\B G)\,,\,\E_n^{\mathrm{h}K}\,]
\end{equation*}
were injective, the analog of Proposition~\ref{prop:injective}. Thus, it will not be easy to generalize the present results to higher chromatic layers.
\end{remark}


\section{Examples and counterexamples}\label{sec:counterexample}

This section contains a couple of examples and counterexamples in order to illustrate the complexity of the subject matter.

\begin{example} As Conjecture~\ref{conj:Ravenel} is trivially true for cyclic groups~$G$, where the Ravenel ideal~$\J_n(G)$ is trivial for all~$n\geqslant1$, we consider the Klein group~$V_4$ with 4 elements all of which are~$2$-torsion.~The Segal conjecture for this particular group has been discussed in~\cite{Davis}.
 
It is easy to see that the ideal~$\J_1(V_4)$ is a free abelian group on one generator. In order to describe a generator, consider the isomorphism class of the $V_4$-set
\begin{equation*}
	\sum_{C\leqslant V_4} V_4/C,
\end{equation*}
where~$C$ runs through the three proper cyclic subgroups of~$V_4$. The associated permutation representation decomposes as a sum of a~$3$-dimensional fixed part and the three different non-trivial~$1$-dimensional~$V_4$-representations. It is therefore isomorphic to the regular~$V_4$-representation plus two fixed lines, which is also the representation associated with the~$V_4$-set
\begin{equation*}
	V_4/e+2V_4/V_4.
\end{equation*}
The difference
\begin{equation*}
	g=\sum_{C\leqslant V_4} [V_4/C]-([V_4/e]+2),
\end{equation*}
generates the kernel of the linearization map, and we will now see that this element gives rise to an essential map
\begin{equation*}
	\Sigma^\infty_+(\B V_4)\longrightarrow\SS\longrightarrow\L_1\SS
\end{equation*} 
at the prime~$2$, so that~$g$ is not in~$\I_1(V_4)$. The calculation presented here is based on~\cite[Section 4]{Laitinen}, which we adapt to the present context. Therefore, the presentation will be brief. 

To start with, let us note that it will be equivalent to show that~$1$ and~\hbox{$1+g$} give rise to different maps. This is slightly more convenient, because the latter element can be written as a product,
\begin{equation}\label{eq:1+g_factorization}
	1+g=\prod_{C\leqslant V_4}([V_4/C]-[V_4/V_4]),
\end{equation}
and each of the factors~$[V_4/C]-[V_4/V_4]$ is the restriction of the transfer element from the group of order two.

It would be easiest if the induced maps could be distinguished in a homology theory. But, since the value of any homology theory on the sphere spectrum is relatively small, it is advisable to pass from spectra to their underlying infinite loop spaces first. On the level of spaces, our candidate~$c=1+g$ becomes a map~$\B V_4\to\Omega^\infty(\SS)$ which we will denote by~$c$ as well. In fact, since~$1+g$ has virtual cardinality~$1+0=1$, the image of~$c$ lies in the component~$\SL_1(\SS)\subseteq\Omega^\infty(\SS)$ of the multiplicative unit. In order to prove that~$g$ is a counterexample to the case~$n=1$ of Ravenel's conjecture, it now suffices to show that the homomorphism induced in a homology theory by the composition
\begin{equation}\label{eq:compositon}
	\B V_4\overset{c}{\longrightarrow}\Omega^\infty(\SS)\overset{e}{\longrightarrow}\Omega^\infty(\L_1\SS)
\end{equation}
with the map~\hbox{$e\colon\Omega^\infty(\SS)\to\Omega^\infty(\L_1\SS)$} induced by the unit
does not agree with the homomorphism induced by the map obtained from the Segal construction on~$1$. Since the latter is trivial in the sense that it induces the zero homomorphism in any reduced homology theory, because it factors over the classifying space of the trivial group, it suffices to show that the composition~\eqref{eq:compositon} induces a non-trivial homomorphism in this sense. 

For the map~$c\colon\B V_4\to\Omega^\infty(\SS)$, this is an elementary finger exercise in mod~$2$ homology, see~\cite[Proposition 4.14]{Laitinen}, but unoriented bordism theory works just as well. This would immediately imply that the composition~$ec$ is non-trivial as well, if only the unit were injective in homology; but it is not. To explain a suitable workaround, let us first discuss the infinite loop space~$\Omega^\infty(\L_1\SS)$ underlying the localization~$\L_1\SS$ of the sphere spectrum~$\SS$ and its relation to classical bundle theory. 

\begin{digression} 
	All the components of the space~$\Omega^\infty(\L_1\SS)$ are equivalent to the space~\hbox{$\JO=\JO(3)$}, which is one of the various spaces classically associated with the image of the~$\J$-homomorphism at the prime~$2$. Recall, for example from the comprehensive discussion in~\cite{Fiedorowicz+Priddy}, that the homotopy fiber~$\overline{\JO}$ of~$\psi^3-\operatorname{id}$ on the~$2$-localization~or~$2$-completion of~$\BO$ has two equivalent components, and these define the space~$\JO$. Its fundamental group is an elementary abelian~$2$-group of rank~$2$, and its universal cover is denoted by~$\widetilde{\J}$: this is also the universal cover of the spaces~$\J$ and~$\JSO$, both of which have fundamental group of order~$2$. The~$2$-connected cover of~$\JO$ is usually denoted by~$\JSpin$. All these spaces come with infinite loop space structures, but those need not concern us here; all that we will use is the identification~$\SL_1(\L_1\SS)\simeq\JO$. This allows us to bring Whitehead's~$j$-map in the form~$j\colon\SO\to\SL_1(\SS)$ into play. 
\end{digression} 

As has already been mentioned, the unit map~$e\colon\Omega^\infty(\SS)\to\Omega^\infty(\L_1\SS)$ is not injective in mod~$2$ homology, not even if we restrict it to suitable components of the spaces, such as~\hbox{$e\colon\SL_1(\SS)\to\SL_1(\L_1\SS)$}. However, there is a substitute: the defining fibration of~$\JO$ loops once to give a fibration
\begin{equation*}
	\SO\longrightarrow \JO\longrightarrow \BO
\end{equation*}
for which the Serre spectral sequences collapses mod~$2$, see~\cite{Clough} and~\cite[Proposition 3.1]{Fiedorowicz+Priddy}. In particular, the inclusion~$d\colon\SO\to \JO=\SL_1(\L_1\SS)$ of the fiber induces an injection in homology. This is our substitute.

The map~$d$ factors over the~$j$-map,
\begin{equation*}
	d\colon\SO\overset{j}{\longrightarrow}\SL_1(\SS)\overset{e}{\longrightarrow}\SL_1(\L_1\SS)=\JO,
\end{equation*}
and so does the transfer~$t\colon\B C\to\SL_1(\SS)\subseteq\Omega^\infty(\SS)$, which is the infinite loop map associated with the Segal map of the transfer element~$[C/e]-[C/C]$ for~$C$:
\begin{equation*}
	t\colon\B C\overset{r}{\longrightarrow}\SO\overset{j}{\longrightarrow}\SL_1(\SS),
\end{equation*}
where the map~$r$ is the reflection map. See~\cite{Brown} and~\cite[Theorem 5.1.b]{Mann+Miller+Miller}. The diagram
\begin{center}
  \mbox{ 
    \xymatrix{
      \B C\ar[r]^-t\ar[d]_-r& \SL_1(\SS)\ar[d]^-e\\
      \SO\ar[r]_-d\ar[ur]|-j & \SL_1(\L_1\SS)
    } 
  }
\end{center}
summarizes the situation. The following diagram chase finishes our discussion of this example. Since~$c$ is the product of restrictions of the transfer along the three projections~$V_4\to C$, it factors as~$c=jb$ for some map~$b\colon \B V_4\to\SO$. Since~$d=ej$ is injective in homology: if~$ec=ejb=db$ were trivial in homology, so were~$b$, and then so were~$c=jb$ as well.
\end{example}


The Klein group $V_4$ with four elements is clearly not the only group for which the conjecture fails. We may use the preceding counterexample to produce some more. For example, the conjecture fails for all elementary abelian~$2$-groups which are not cyclic, and more generally for all groups~$G$ which have~$V_4$ as a retract. The following example is not of this form.

\begin{example}
	The alternating group~$A_4$ on four symbols has twelve elements which fall into four conjugacy classes. There are five conjugacy classes of subgroups in~$A_4$, determined by the orders of their representatives:~$e$,~$C_2$,~$C_3$ of order three, the Klein group~$V_4$, and~$A_4$ itself. For example, the isomorphism class~$[A_4/C_3]$ is represented by the defining~$A_4$-set with four symbols on which the elements of~$A_4$ act as even permutations. An easy computation shows that~$\J_1(A_4)$, the common kernel of~$\Phi^e$,~$\Phi^{C_2}$, and~$\Phi^{C_3}$, which as an abelian group is free of rank~2, is generated by the two elements
\begin{equation}
	[A_4/e]-3[A_4/C_3]-[A_4/V_4]+3
\end{equation}
and
\begin{equation}\label{eq:Acandidate}
	[A_4/C_2]-[A_4/C_3]-[A_4/V_4]+1.
\end{equation}
The restriction of the second element~\eqref{eq:Acandidate} from~$A_4$ to~$V_4$ is the element~$g$ considered above. As this has already been shown to give rise to an essential map, so does~\eqref{eq:Acandidate}. Therefore, Conjecture~\ref{conj:Ravenel} fails for~$A_4$ as well.
\end{example}


\begin{example}
	The restriction of the element~$g$ along the (essentially unique) projection~$q\colon D_8\to V_4$ from the dihedral group~$D_8$ of order~$8$ to~$V_4$ does behave according to Conjecture~\ref{conj:Ravenel}: the composition
\begin{equation}\label{eq:D8_compositon}
		\Sigma^\infty_+(\B D_8)\longrightarrow
		\Sigma^\infty_+(\B V_4)\longrightarrow
		\SS\longrightarrow\L_1\SS
\end{equation}
	is zero at all primes. (Of course, the only prime of interest here is~$p=2$.) One way to see this is as follows. Note that
\begin{equation*}
		q^*g=[D_8/W_4]+[D_8/W_4']+[D_8/C_4]-[D_8/C_2]-2,
\end{equation*}
where~$W_4$,~$W_4'$, and~$C_4$ are the subgroups of~$D_8$ of order~$4$, and~$C_2$ is the normal subgroup of order~$2$. The image of this virtual~$D_8$-set under the Segal map extends over~$\Sigma^\infty_+(\BU)$, where~$\BU$ is the classifying space of the stable unitary group. Compare~\cite[Proposition~4]{Feshbach}. In order to show that the composition~\eqref{eq:D8_compositon} is zero, one may again replace~$\B V_4$ from~\eqref{eq:D8_compositon} by~$\BU$, pass to infinite loop spaces, and show that the composition
\begin{equation*}
		\B D_8\longrightarrow
		\BU\longrightarrow
		\Omega^\infty(\SS)\longrightarrow\Omega^\infty(\L_1\SS)
\end{equation*}
is null. To show this, it suffices to observe that all maps~$\BU\to\Omega^\infty(\L_1\SS)$ are zero. This is so, because~$\BU$ is simply-connected, so that any such map factors thought a connective cover of~$\Omega^\infty(\L_1\SS)$, which again is one of the various spaces classically associated with the image of the~$\J$-homomorphism at the prime~$2$. By~\cite[Lemma 2]{Feshbach}, there are no essential maps from~$\BU$ to these.
\end{example}



\vfill

\parbox{\linewidth}{%
Markus Szymik\\
Department of Mathematical Sciences\\
NTNU Norwegian University of Science and Technology\\
7491 Trondheim\\
NORWAY\\
\href{mailto:markus.szymik@ntnu.no}{markus.szymik@ntnu.no}\\
\href{https://folk.ntnu.no/markussz}{folk.ntnu.no/markussz}}

\end{document}